\RequirePackage{fix-cm}
\documentclass[smallextended]{svjour3}       % onecolumn (second format)
\smartqed  % flush right qed marks, e.g. at end of proof
\usepackage{latexsym}
\usepackage{amsmath}

\begin{document}
\begin{center}
\textbf{\Large New Directions for Primality Test}\\[0.5cm]
\textbf{Lakshmi Prabha S}$^{a}$         $\bullet $      \textbf{T.N.Janakiraman}$^{b}$ \\
Department of Mathematics, National Institute of Technology,\\
Trichy-620015, Tamil Nadu, India.\\
Emails: $^{a}$jaislp111@gmail.com, $^{b}$janaki@nitt.edu\\[2cm]
\end{center}

\begin{abstract}
In this paper, two approximation algorithms are given. Let $N$ be an odd composite number. The algorithms give  new directions regarding primality test of given $N$. The first algorithm is given using a new method called  digital  coding method. It is conjectured that the algorithm finds a  divisor of $N$ in at most $O(\ln^{4} N)$, where $\ln$ denotes the logarithm with respect to base 2. The algorithm can be applied to find the next largest Mersenne prime number. Some directions are given regarding this. The second algorithm finds a prime divisor of $N$ using the concept of graph pairs and it is proved that the time complexity of the second algorithm is at most $O(\ln^{2} N)$ for infinitely many cases (for approximately large $N$). The advantages and disadvantages of the  second algorithm are also analyzed. 

\end{abstract}

\section{Introduction}

First, let us give some basic definitions of number theory from Apostol~\cite{NT}, which are used in this paper.\\ 

We say that an integer $d$ divides an integer  $n$ and write $d|n$ whenever $n=cd$ for some integer $c$. If $d$ divides two integers $a$ and $b$, then $d$ is called a \textit{common divisor} of $a$ and $b$.
\begin {definition} \textbf{G.C.D. of two numbers} Let $a$ and $b$ be two numbers both not zero. Then the Greatest Common Divisor (G.C.D.) of $a$ and $b$ is a number $d$ such that
\begin{itemize}
	\item $d$ is a common divisor of $a$ and $b$, and 
\item every common divisor of $a$ and $b$ divides this $d$. 
\end{itemize}
Let us denote it by $gcd(a,b)$. If $gcd(a,b)=1$, then $a$ and $b$ are said to be relatively prime.
\end{definition}
\begin{definition}
An integer $N$ is called prime if $N > 1$ and if the only positive divisors of $N$ are 1 and $N$. If $N > 1$ and $N$  is not prime, then $N$ is called composite.\end{definition}
\begin{definition} Given integers $a,b,m$ with $m > 0$. Then $a$ is said to be congruent to $b$ modulo $m$, if $m$ divides the difference $(a-b)$. This is denoted by $a \equiv b($mod $m)$. The number $m$ is called the modulus of the congruence.
 \end{definition}

We list some of the following basic theorems in  number theory from Apostol~\cite{NT}, which are used for our results.\\
\textbf{Properties:} 
\begin{theorem}
\label{product}
If $a \equiv b ($mod $m)$ and $\alpha \equiv \beta ($mod $m)$, then $a\alpha \equiv b\beta ($mod $m)$.
\end{theorem}
\begin{theorem}
\label{cancel}
$ac \equiv bc ($mod $mc)$ iff $a \equiv b ($mod $m)$.
\end{theorem}
\begin{theorem}
\label{Fermat}
Fermat's Theorem: If a prime $p$ does not divide \lq $a$\rq,\\ then $a^{(p-1)} \equiv 1 ($mod $p)$.
\end{theorem}
For more details on basics of number theory, readers are directed to refer Apostol~\cite{NT}.\\
The remaining part of the paper is organized as follows: 
\begin{itemize}
\item Section~\ref{sec2} discusses about prior work.
	\item Section ~\ref{sec3} contains an interesting approximation algorithm, which is given using a new approach.The intricate portions of the algorithm are discussed and a conjecture is given regarding the time complexity of the  algorithm.
\item In section ~\ref{sec4}, a simple approximation algorithm is given using the concept of graph pairs. Time complexity and correctness of that algorithm are given. The advantages and disadvantages of that algorithm are also discussed in this section. Further a procedure is given using that algorithm for primality test and an open problem is given regarding its time complexity.
\end{itemize}
\section{Prior Work}
\label{sec2}
A primality test is an algorithm for determining whether an input number $N$ is prime. It has a lots of applications in cryptography and network security. For centuries, number theory was considered to be the most pure form of Mathematics, because there were no practical applications, as far as anyone could tell. However, in the latter half of the 20th century, number theory became
central to developments in digital security, for example,  public key cryptography, 
credit card check digits and so on. If one is capable  to quickly factor an integer into a product of two large primes and verify that they were both primes, then he would be able to break
into most banking systems.
 
Most popular primality tests are probabilistic tests. These tests use, apart from the tested number $N,$ some other numbers $a$ which are chosen at random from some sample space; The simplest is Fermat Primality test. Miller- Rabin primality test and Solovay and Strassens primality test are also probabilistic tests. Miller~\cite{Mill} in 1975, Rabin~\cite{Rabin} in 1980 and Solovay and Strassens ~\cite{SS} in 1974, gave randomized algorithms. Their method can be made deterministic under the assumption of Extended Riemann Hypothesis (ERH). Since then, a number 
of randomized polynomial-time algorithms have been proposed for primality testing, based on many
different properties.

In 1983, Adleman, Pomerance, and Rumely~\cite{APR} achieved a major breakthrough by giving a deterministic algorithm for primality that runs in \\$(\ln N)^{O(\ln \ln \ln N)}$ time (all the previous deterministic algorithms required exponential time). 

Then, in 1986, Goldwasser and Kilian~\cite{GK} and Atkin~\cite{ATK} also  gave randomized algorithms, based on elliptic curves. In 1992, Adleman and  Huang~\cite{Huang} modified the Goldwasser-Kilian method and presented an errorless (but expected polynomial-time) variant of the elliptic curve primality test. But it is also a randomized algorithm that runs in expected polynomial-time on all inputs.

In 2002, the first provably polynomial time test for primality was invented by  Agrawal, Kayal and  Saxena~\cite{AKS}. They proved that AKS primality test, runs in $O(\ln^{7.5} N)$ and  which can be further reduced to $O(\ln^{6}N)$ if the Sophie Germain conjecture is true. 

The aim of this paper is to reduce this time complexity. We have given two approximation (randomized) algorithms. They do not deterministically distinguish whether the given number is prime or not in polynomial time. But the second algorithm gives good results for infinitely many cases, (not for all inputs). It is proved that our second  algorithm finds a prime divisor of $N$ in at most $O(\ln^{2} N)$ for infinitely many cases (for approximately large $N$).
  
\section{Digital Coding Algorithm}
\label{sec3}
In this section, we give an interesting approximation algorithm, the procedure of which is new and different. The algorithm finds a divisor of $N$.  We have employed binary and decimal equivalent of the numbers. It is observed that the time complexity of the algorithm is at most $O(\ln^{4}N)$. The correctness of the algorithm is verified (manually) for $N \leq 10^{5}$. So, we pose the conjecture that the algorithm finds a divisor of $N$, in at most $O(\ln^{4}N)$, for any odd composite $N$. Also, we pose some other conjectures supporting this conjecture. 

Some applications of the algorithm are  also mentioned, mainly, there are higher possibilities to find the next largest Mersenne prime number using this method.
\begin{note} In this section (alone), the notation \lq $|$\rq \ does not denote \lq divides\rq \  symbol. The notation \lq $|$\rq  \ denotes the partition of a number (in this section alone). For example, $4|5$ does not mean 4 divides 5. But it means that the number 45 is partitioned digit-wise.
\end{note}
\begin{definition}
\textbf{Digital Binary Equivalent} A binary number $B$ is said to be digital binary equivalent of the decimal number $A$, if $B$ is obtained from $A$ by finding binary equivalents of each digit of $A$ (starting from left and ending at right) and then taking union of those binary equivalents.\\
For example, consider $A = 872$. First, let us find binary equivalent of each digit of $A$.
Binary equivalent of 8 is 1000; 7 is 111 and that of 2 is 10. Next, we take union of these binary numbers and write it as: 100011110. So, digital binary equivalent of 872 is 100011110. Let us represent this process in the following form: $8|7|2 \stackrel{dig-bin}{\rightarrow} 1000|111|10$.
\end{definition}
\subsection{Notation}
The following notation is used throughout this section~\ref{sec3}.
\begin{itemize}
\item Let $N = n_1 n_2 \ldots n_r$, where $n_i$ is the digit of the number $N$ and $r$ is the total number of digits of $N$.
\item $B_j$ denotes a binary number. It is obtained by finding digital binary equivalent of a decimal number $N$ or $D_j$.
\item $D_j$ denotes a decimal number. It is obtained by finding decimal equivalent of the binary number $B_j$.
\item $A \stackrel{decimal}{\rightarrow} B$ denotes that the number $A$ is in binary form,  the number $B$ is in decimal form and $B$ is obtained from $A$ by finding decimal equivalent of $A$.
\item $A \stackrel{dig-bin}{\rightarrow} B$ denotes that the number $A$ is in decimal form,  the number $B$ is in binary form and $B$ is obtained from $A$ by finding digital binary  equivalent of $A$.
\item $A \stackrel{dig-bin,zero}{\rightarrow} B$ denotes that the number  $A$ is in decimal form,  the number $B$ is in binary form.  $B$ is the digital binary  equivalent of $A$ with some extra zeros, that is, first we find the digital binary equivalent of $A$, say $A_1$ and then zeros are inserted in between the binary equivalents of some digits of $A$ (in order to get result). This digital binary equivalent with extra zeros is the result $B$.  
\item $A \stackrel{zero}{\rightarrow} B$ denotes that the numbers $A$ and $B$ are in binary form and $B$ is obtained by inserting zeros in between the digits of $A$ (in order to get result). 
\end{itemize}
\subsection{APPROXIMATION ALGORITHM}
In this section, an approximation algorithm is given using  digital binary coding method, to find a divisor of odd composite $N$.\\
Input: $N$.\\ 
Output: \lq\lq$N$ is composite\rq\rq\, and $gcd(D_j,N)$, which is a divisor of $N$, if $N$ is composite.\\
No output is produced if $N$ is prime.\\
Algorithm 1:\\
JRLP($N$)
\begin{enumerate}
\item Initially let $j = 0$.\\
DBC($N$)
\item	Let $N = n_1 n_2 \ldots n_r$, where $n_i$ is the digit of the number $N$ and $r$ is the total number of digits of $N$.
\item Find the digital binary equivalent of $N$.
\item Put $j = j+1$ and store the digital binary number in $B_j$. 
\item Find the decimal equivalent of $B_j$ and store it in $D_j$.
\item If $gcd(D_j, N) \neq 1$, then print \lq\lq$N$ is composite\rq\rq\, and print $gcd(D_j,N)$ . 
\item else if (the number of digits of $D_j$) $r \neq 1$ or the number $D_j$ does not occur in any of the previous steps, call DBC$(D_j)$.
\item else GOTO next step. //At this stage, either $r$ becomes $1$ or a number $D_j$ is repeated, but still $gcd(D_j, N) = 1$. Initially, we have obtained  a sequence of binary and decimal equivalents starting from $N$. Let us call this initial sequence as the first chain of $N$. Next, we backtrack and find next chain of $N$.
\item Backtrack one step before in the first chain of $N$ and insert extra zeros \textbf{in between the binary equivalents} of $n_i$ (refer example 2) and then GOTO step 4. 
\item	Repeat backtracking (step 9) until we reach the starting point $N$.
\item If the process of inserting zeros and the remaining entire process (from steps 2 to 7) is over for $N$, then stop.// At this last stage, $gcd(D_j, N) = 1$.
\end{enumerate}
\textbf{Why the name Digital Coding?} \\
We name this method as \lq\lq DIGITAL CODING METHOD\rq\rq, because we give binary coding for each digit of the decimal number, not for the entire decimal number. \subsection{EXPLANATION WITH ILLUSTRATIONS}
\textbf{Example 1:} Take $N = 88837$.\\
$8|8|8|3|7 \stackrel{dig-bin}{\rightarrow} 1000|1000|1000|11|111$ (Digital coding of 88837)\\ 
Now, $B_1 = 10001000100011111 \stackrel{decimal}{\rightarrow} 69919 = D_1$ (Decimal equivalent of $B_1$).  As, $gcd(D_1,N) = 1$, we continue this procedure. The step by step executions of the whole process of Algorithm 1 is given in the  Table~\ref{tab:1}.
\begin{table}[h]
\caption{Step by Step Executions of Algorithm 1}
\label{tab:1}      
\begin{tabular}{lllll}
\hline\noalign{\smallskip}
$j$	& $D_{j-1}$	& $B_j$	&  $D_j$	&  $gcd(D_j,N)$\\
\noalign{\smallskip}\hline\noalign{\smallskip}

1 & 	88837 = $N$	& 10001000100011111 &	69919 &	1 \\
2	& 69919	& 1101001100111001	& 54073 &	1\\
3	& 54073 &	101100011111	& 2847	& 1 \\
4	& 2847	 & 101000100111 &	2599	& 1 \\
5	& 2599	& 1010110011001	& 5529	& 1\\
6	& 5529 &	101101101001	& 2921	& 1\\
7	& 2921	& 101001101	& 333	& 37\\
\noalign{\smallskip}\hline
\end{tabular}
\end{table}
\\ \\ \textbf{Output:}\lq\lq$N$ is composite\rq\rq\ and $gcd(D_j,N) = 37$.\\
\textbf{Example 2:} Take $N = 15$.\\
 $1|5 \stackrel{dig-bin}{\rightarrow}  1|101 \stackrel{decimal}{\rightarrow} 1|3 \stackrel{dig-bin}{\rightarrow}  1|11  \stackrel{decimal}{\rightarrow} 7$.\\This sequence is called the \textbf{first chain} of $N$.\\
(13 is the decimal equivalent of 1101. Then 13 is digitally coded. 7 is the decimal equivalent of 111).\\
\textbf{Backtrack 1:}   $1|3 \stackrel{dig-bin,zero}{\rightarrow}  1|011 \stackrel{decimal}{\rightarrow}  11 $. \\ This sequence is called \textbf{second chain} of $N$.(Originally the digital coding  for the number 13 is 111, that is, $1|3 \stackrel{dig-bin}{\rightarrow} 1|11$. Now, we have inserted zeros in between the two digital codings, that is, between 1 and 11 (or in front of 11)). \\
Next we backtrack in the backtrack 1.\\
\textbf{Backtrack 2:} Consider $1|1 \stackrel{zero}{\rightarrow} 101 \stackrel{decimal}{\rightarrow} 5$.\\ This sequence is called \textbf{third chain} of $N$ and we get $gcd(5,15) = 5$. 
\subsection{ NP-COMPLETENESS}
The process of inserting zeros is NP-Complete.\\
\textbf{Explanation:} There is no proper rule in inserting zeros. In the algorithm, first we do the procedure without inserting any extra zeros. If result does not come, we do backtracking and then insert zeros. Even at the time of back tracking also, there are three possibilities.
\begin{enumerate}
\item We have to insert zeros at some stage and need not insert zeros at some other stage so that we  get correct result. 
\item Consider one particular stage (only one number). We may have to insert zeros for some digit of a number and need not insert zeros for some other digit of the number so that we get correct result.
\item We can insert zeros uniformly for all digits at a particular stage  and uniformly for all numbers so that we get correct result.
\end{enumerate}

So, we do not uniformly insert zeros. The obvious question is among these possibilities, which has to be used. The questions are: 
\begin{itemize}
	\item Where to insert zeros? 
	\item When to insert zeros? and 
\item How to insert zeros?
\end{itemize}

Suppose we assume that we insert zeros with a rule. \\
\textbf{Rule: Equally expanding:} Add zeros to make digits same. \\
For example, if $N$ = 51, then $B_1$ is 1011, that is, $5|1 \stackrel{dig-bin}{\rightarrow} 101|1$. Here 101 is of three digits and 1 is of 1 digit. We can code 1 as 01, 001, 00001 and so on. But, we make 1 to be of three digits and hence we code  1 as 001. Now, 51 can be coded as 101001, that is, $5|1 \stackrel{dig-bin,zero}{\rightarrow} 101|001$.   This is called \textbf{equally expanding method}. 

The surprising thing is sometimes the result occurs in few steps, but unknowingly, we would have tried many steps. \\ For example, consider $N=451$.\\
$4|5|1 \stackrel{dig-bin}{\rightarrow} 100|101|1 \stackrel{decimal}{\rightarrow} 7|5 \stackrel{dig-bin}{\rightarrow} 111|101 \stackrel{decimal}{\rightarrow} 6|1 \stackrel{dig-bin}{\rightarrow} 110|1 \stackrel{decimal}{\rightarrow} 1|3 \stackrel{dig-bin}{\rightarrow} 1|11 \stackrel{decimal}{\rightarrow} 7$.\\
\textbf{Backtrack:} $13 \stackrel{dig-bin,zero}{\rightarrow} 1|011 \stackrel{decimal}{\rightarrow} 11$ and we get $gcd(451,11) = 11$. This process takes 5 iterations.\\
Now, let us try using the equally expanding rule.\\
$4|5|1 \stackrel{dig-bin, zero}{\rightarrow} 100|101|001 \stackrel{decimal}{\rightarrow} 297$ and we get $gcd(451,297) = 11$. This process takes only 1 iteration.

So, there is no uniformity in this method.This leads to the NP-Completeness of the process. But we observed that if $N$ is composite, then we will surely get the divisor of $N$. So, we say that the probability of getting a divisor of odd composite $N$ using digital  coding method is always 1.That is,\\ P(getting divisor of odd composite $N$ using  digital coding method) = 1. But how it can be obtained is NP-complete.

As we do the back tracking step until we get result, this method of backtracking and inserting zeros is equivalent to trial and error method, unless there is a proper rule to do it. But there is a hidden rule to achieve it easily. 
\subsection{TIME COMPLEXITY}
It is already mentioned in the algorithm that at the end of step 8, one chain of $N$ will be formed, i.e, from steps 2 to 8, one chain will be formed. One iteration of the algorithm is defined as executing steps 2 to 6 one time. The time complexity of the algorithm lies in two main steps:
\begin{enumerate}
	
\item Finding $gcd(D_j, N)$ 
\item How many times $gcd(D_j,N)$ is found in the algorithm (Number of iterations of the algorithm) and
\item The number of backtracking steps.
\end{enumerate}

We find $gcd(D_j, N)$ at each iteration of the algorithm. Each computation of G.C.D. takes $O(\ln N)$ time~\cite{GCD}. Next, we do not know what is the number of iterations, that is, number of times $gcd(D_j,N)$ is found out in the algorithm. Also, we do not know what is the number of backtracking steps \textbf{with a proper rule of inserting zeros}, in the algorithm.

While verifying (manually) the algorithm for $N \leq 10^{5}$, we found that the total time complexity of the algorithm is at most $O(\ln^{2}N)$. We checked for some higher cases of $N$. We observed that in some cases, it directly gives result (without inserting zeros) within $O(\ln^{4}N)$ and in some cases, extra zeros have to be inserted and if zeros are \textbf{properly inserted}, it gives result within $O(\ln^{4}N)$. So, we give the following conjectures.\\[10pt]
\textbf{CONJECTURE 1:}
\textbf{JRLP's Conjecture:}\\
The total time complexity of the Algorithm 1 is at most $O(\ln^{4}N)$. This should be supported by the following statement:\\
The number of iterations and the number of backtracking steps \textbf{with a proper rule of inserting zeros} in the algorithm take at most $O(\ln^{4}N)$.\\[10pt]
\textbf{CONJECTURE 2:} P(getting a divisor of odd composite $N$ using  digital coding method) = 1.
\subsection{APPLICATIONS}
The primes of the form $M_p = 2^{p} - 1$, where $p$ is a prime, are called Mersenne primes. Let us call $p$ to be the generator prime of $M_p$.

We can generate Mersenne primes using this concept as follows: We applied JRLP algorithm (without inserting zeros)  in all  the generator primes, $p$, of the existing Mersenne primes $M_p$. We observed that the first chain (sequence) of the generator  primes, $p$, is terminated with either 11 or 7 or 47 or 5 or  9. So, Mersenne primes $M_p$ can be partitioned into five groups, based on the terminating number of the  first chain of the generator prime $p$. We observed that the first chain of  the generators of recent Mersenne primes are terminated with 11 and also there are many such Mersenne primes. So, we start form 11, do the converse process of the Algorithm 1 and try to get some generator primes and hence the  Mersenne primes. Similarly, we can start from any of the above mentioned numbers and try to get the generator primes and hence the  Mersenne primes.   The following are some examples, illustrating how the first chain of the  generator primes terminated with 11 or 7 or 47 or 5 or 9.
\begin{itemize}
	\item $8|9 \stackrel{dig-bin}{\rightarrow} 1000|1001 \stackrel{decimal}{\rightarrow} 1|3|7 \stackrel{dig-bin}{\rightarrow} 1| 11| 111 \stackrel{decimal}{\rightarrow} 6|3 \stackrel{dig-bin}{\rightarrow} 110| 11 \stackrel{decimal}{\rightarrow} 2|7 \stackrel{dig-bin}{\rightarrow} 10 |111 \stackrel{decimal}{\rightarrow} 2|3 \stackrel{dig-bin}{\rightarrow} 10| 11 \stackrel{decimal}{\rightarrow} 11$. (As we know that $N$ is prime, we need not backtrack and proceed further).
\item $1|7 \stackrel{dig-bin}{\rightarrow} 1| 111 \stackrel{decimal}{\rightarrow} 1|5 \stackrel{dig-bin}{\rightarrow} 1| 101 \stackrel{decimal}{\rightarrow} 1|3 \stackrel{dig-bin}{\rightarrow} 1|11 \stackrel{decimal}{\rightarrow} 7$.
\item	$1|2|7|9 \stackrel{dig-bin}{\rightarrow} 1 |10 |111| 1001 \stackrel{decimal}{\rightarrow} 8|8|9 \stackrel{dig-bin}{\rightarrow} 1000| 1000| 1001 \stackrel{decimal}{\rightarrow} 2|1|8|5 \stackrel{dig-bin}{\rightarrow} 10 |1| 1000| 101 \stackrel{decimal}{\rightarrow} 7|0|9 \stackrel{dig-bin}{\rightarrow} 111| 0| 1001 \stackrel{decimal}{\rightarrow} 2|3|3 \stackrel{dig-bin}{\rightarrow} 10| 11| 11 \stackrel{decimal}{\rightarrow} 47$.
\item $1|9|9|3|7 \stackrel{dig-bin}{\rightarrow} 1| 1001| 1001| 11| 111 \stackrel{decimal}{\rightarrow} 1|3|1|1|9 \stackrel{dig-bin}{\rightarrow} 1| 11| 1|1| 1001 \stackrel{decimal}{\rightarrow} 5|0|5 \stackrel{dig-bin}{\rightarrow} 101| 0| 101 \stackrel{decimal}{\rightarrow} 8|5 \stackrel{dig-bin}{\rightarrow} 1000| 101 \stackrel{decimal}{\rightarrow} 6|9 \stackrel{dig-bin}{\rightarrow} 110| 1001 \stackrel{decimal}{\rightarrow} 1|0| 5 \stackrel{dig-bin}{\rightarrow} 1| 0| 101 \stackrel{decimal}{\rightarrow} 2|1 \stackrel{dig-bin}{\rightarrow} 10| 1 \stackrel{decimal}{\rightarrow} 5$.
\item $2|2|0|3 \stackrel{dig-bin}{\rightarrow} 10| 10| 0| 11 \stackrel{decimal}{\rightarrow} 8|3 \stackrel{dig-bin}{\rightarrow} 1000| 11 \stackrel{decimal}{\rightarrow} 3|5 \stackrel{dig-bin}{\rightarrow} 11 |101 \stackrel{decimal}{\rightarrow} 2|9 \stackrel{dig-bin}{\rightarrow} 10| 1001 \stackrel{decimal}{\rightarrow} 4|1 \stackrel{dig-bin}{\rightarrow} 100|1 \stackrel{decimal}{\rightarrow} 9$.

\end{itemize}
Table~\ref{tab:2} shows the Mersenne primes $M_p$ partitioned in to five categories, based on the  terminating number of the first chain of their generators $p$.
\begin{table}[h]
\caption{Partition of Mersenne Primes}
\label{tab:2}      
\begin{tabular}{lllll}
\hline\noalign{\smallskip}
 Terminating \\with 11 &	 7 &	 47	&  5 &	9\\
\noalign{\smallskip}\hline\noalign{\smallskip} 
3 & 7 &    1279 &       5 & 2203 \\
89 &	17 & 3217	 &	 19937 &	9689\\
107	& 31 &	4253 &	  44497 \\
607 &	61	& 9941 &	  216091\\	
21701 &	127 &	11213 &	  756839\\	
110503 & 521	& 859433  &	25964951	\\
1257787	& 2281 & 2976221			\\
1398269 &	4423 &	3021377		\\
20996011 &	23209 &	13466917 \\		
24036583 &	86243 \\
32582657 &	132049			\\
37156667 &	6972593			\\
42643801 &	30402457			\\
43112609\\				
57885161	\\			
\noalign{\smallskip}\hline
\end{tabular}
\end{table}
\\[1cm] 
\textbf{CONJECTURE 3:} Mersenne primes can be partitioned into five groups only.\\
\subsubsection{Rough Method to generate  primes}
\begin{definition}
\textbf{Cell of a binary number} Let $A$ be the given binary number. Let it be of the form $A = r_1r_2\ldots r_t$, where $r_i$ is a digit of $A$ and $t$ is the total number of digits of $A$.
Next, let us partition the digits of $A$ as follows: $r_1r_2|r_3|r_4r_5r_6|r_7$. Then each partitioned (divided) portion of $A$ is called the \textbf{cell} of $A$. Here, $r_1r_2$ is called the first cell of $A$, $r_3$ is the second cell of $A$ and so on.
\end{definition}
\begin{definition}
\textbf{Cell-wise Decimal Equivalent} A decimal number $B$ is said to be cell-wise decimal equivalent of the binary number $A$, if $B$ is obtained from $A$ by finding decimal equivalents of each cell of $A$ (starting from left and ending at right) and then taking union of those decimal equivalents.

For example, consider $A = 1011$. First, let us find one partition of $A$. $A$ can be partitioned as $10|11$. Here, $10$ is the first cell of $A$ and 11 is the second cell of $A$. The decimal equivalnt of 10 is 2 and that of 11 is 3. Union of those decimal equivalents is 23. Thus, the cell-wise decimal equivalent of $A$ is 23. Let it be represented as follows:\\
$10|11\stackrel{cell-dec}{\rightarrow} 2|3$.\\ 
Find all the partitions of $A$. They are: $1|0|1|1$, $1|0|11$, $10|11$ and $101|1$.  The cell-wise decimal equivalents of each of the partitions of  $A$ are as follows: 
\begin{itemize}
	\item $1|0|1|1 \stackrel{cell-dec}{\rightarrow} 1|0|1|1$
	\item $1|0|11\stackrel{cell-dec}{\rightarrow} 1|0|3$, where 3 is the decimal equivalent of the binary number 11.
	\item $10|11\stackrel{cell-dec}{\rightarrow} 2|3$, where 2 is the decimal equivalent of the binary number 10 and 3 is the decimal equivalent of the binary number 11.
	\item $101|1\stackrel{cell-dec}{\rightarrow} 5|1$, where 5 is the decimal equivalent of the binary number 101.
\end{itemize} 
\end{definition}
\textbf{Notation}\\
The following notation is used for the procedure.
\begin{itemize}
\item $T$ denotes the starting number of the procedure and $T$ = 11 or 7 or 47 or 5 or 9.
\item $A$ denotes a binary number and it is obtained by finding the binary equivalent of $T$ or $E$.
\item $E$ denotes a decimal number and it is obtained by finding cell-wise decimal equivalent of $A$.
\item $A \stackrel{binary}{\rightarrow} B$ denotes that the number $A$ is in decimal form,  the number $B$ is in binary form and $B$ is obtained from $A$ by finding binary equivalent of $A$.

\item $A \stackrel{cell-dec}{\rightarrow} B$ denotes that the number $A$ is in binary form,  the number $B$ is in decimal form and $B$ is obtained from $A$ by finding cell-wise decimal equivalent of $A$.
\end{itemize}
\textbf{ROUGH PROCEDURE}\\
We start with any of the numbers 11, 7, 47, 5 and 9. Let $T$ = 11 or 7 or 47 or 5 or 9. \\
\textbf{Procedure:}\\
generate($T$)
\begin{enumerate}
\item Find the binary equivalent of $T$ and store it in $A$.
\item Find all the partitions of $A$.
\item \textbf{for} each partition of $A$,\\
$\{$
\begin{itemize}
	\item  Find cell-wise decimal equivalent of  $A$ and store it in $E$.
\item Call generate($E$) until we get the next largest prime number or Mersenne prime number.
\end{itemize}
 $\}$
\end{enumerate}
\textbf{ILLUSTRATION}\\Consider $T$ = 11.\\
Iteration 1: $11 \stackrel{binary}{\rightarrow}  1011$. Let $A =1011$.\\
All the partitions of $A= 1011$ are $1|0|1|1$, $1|0|11$, $10|11$ and $101|1$.\\
Consider a partition $1|0|11$. $1|0|11 \stackrel{cell-dec}{\rightarrow}1|0|3$. Let $E = 103$. Now, $2^{103}-1$ is not prime and hence 103 is not the generator prime. So, we continue the procedure.\\
Iteration 2: Consider $E = 103$.\\
$103 \stackrel{binary}{\rightarrow} 1100111$. Let $A = 1100111$.\\
Find all the partitions of $A$.\\
Let us consider  a partition $1|10|0|111$. $1|10|0|111 \stackrel{cell-dec}{\rightarrow} 1|2|0|7$.
Let $E = 1207$.\\
Next go to iteration 3 and so on.\\
 
Thus, this method may help to find the next largest prime number.
\section{Different Approach Algorithm using the concept of Graph Pairs}
\label{sec4}

There is a strong connection between number theory and graph theory. An example for such a connection is the \textit{cycle graph}, discussed in Anderson~\cite{cycle}. Many works are done connecting these two concepts.  Janakiraman and Boominathan~\cite{TNJ1,TNJ} have brought out a different kind of interplay between these two concepts. Using the theory of congruences, Janakiraman and Boominathan~\cite{TNJ1,TNJ} have defined graph pairs, simple graph pairs, the graph corresponding to a graph pair and so on. 

In this section, we use the concept of graph pairs to find a prime divisor of $N$. This paper deals only with the \textit{number theoretic concepts of graph pairs}.  For more details of graph pairs and its interplay between number theory and graph theory, readers are directed to refer Janakiraman and Boominathan~\cite{TNJ}. 

In this section,  a simple approximation algorithm is given for finding a prime divisor of $N$ using the concept of  graph pairs of $N$. Its time complexity is proved to be at most $O(\ln^{2}N)$, for infinitely many $N$ (also approximately large $N$). Then the disadvantages and advantages of the algorithm are analyzed. Also, using this algorithm, another procedure is given in this section to check up  whether the given number is prime or not. The time complexity of that procedure is discussed and an open  problem is given regarding the time complexity.

In this section, the notation $d|n$ denote $d$ divides $n$.
\subsection{DEFINITIONS}
Next let us give the definitions of graph pair and simple graph pair from Janakiraman and Boominathan~\cite{TNJ1,TNJ}.
\begin{definition} \textbf{Graph Pair}
 Given a positive integer $N$, the pair $(a,b)$ of positive integers is defined to be a graph pair if $a.b \equiv 1($mod $N)$. If $a^{2} \equiv 1 ($mod $N)$, then $(a,a)$ is defined as a \textbf{simple graph pair}.
 \end{definition}
 
In number theory terminology, $b$ is called the inverse of $a$ modulo $n$. If $gcd(a,n) = 1$, then $a$ has an inverse, and it is unique modulo $n$. 

$(a,b)$ is named as a graph pair because for  the graph pair $(a,b)$ of $N$, there is a graph on $(N-1)$ vertices. Let us discuss in detail as follows:\\ 
For a given positive integer $N$, let $(a,b)$ be a graph pair. Form a matrix $A = \left[a_{ij}\right]$ as follows:
\begin{equation*}
a_{ij} = 
\begin{cases}
1 & \text{ if either $i.a \equiv j ($mod $N)$ or $i.b \equiv j$ (mod $N$)}\\ 
$0$ & \text{otherwise, $i = 1,\ldots, N-1$.}
\end{cases}
\end{equation*}

Since $a$ and $b$ are relatively primes to $N$ and $i$ = $1,2,\ldots, N-1$; $j$ takes values from $1$ to $N-1$ only.
This always results in a symmetric binary square matrix $A$ of order $N-1$.
Since there is a one-to-one correspondence between a labeled graph on $N$ vertices and a $N \times N$ symmetric binary matrix, $(a,b)$ is named as a graph pair, that is, for  the graph pair $(a,b)$ of $N$, there is a graph on $(N-1)$ vertices, whose adjacency matrix is nothing but $A$.

\subsection{APPROXIMATION ALGORITHM}
\label{Algo3}
 In this section, an approximation algorithm is given to find a prime divisor of $N$.\\
Input: $N$\\
Output: \lq\lq$N$ is composite\rq\rq, $gcd(D,N)$, and $j$, if $N > 2^{j}$.\\
No output is produced if $N < 2^{j}$.\\
Algorithm 2:\\
GP($N$)
\begin{enumerate}
\item Let $j = 0$.
\item Put $j = j+1$.
\item If $(N-1)/2^{j}$ is an integer, then find the graph pair $((N-1)/a, N-a)$ such that $a = (N-1)/2^{j}$.
\item else find the graph pair $(2^{(j-1)}. 2, (b_{j-1}.b_1)($mod $N))$.
\item	Denote this graph pair as $(2^{j}, b_j)$.
\item	Find $|b_j- 2^{j}|$ and store it in $D$.
\item If $2^{j} < N$, then GOTO step 8 else stop.
\item If $gcd(D,N) = 1$, then GOTO step 2 else print \lq\lq$N$ is composite\rq\rq, print $gcd(D,N)$ and $j$ and stop.
\end{enumerate}
\begin{note}
\label{other}
 We have used two methods to generate the graph pairs, namely, $(2^{j}, b_j) = ((N-1)/a, N-a)$ such that $a = (N-1)/2^{j}$  and  $(2^{j}, b_j) = (2^{(j-1)}. 2, (b_{j-1}.b_1)($mod $N))$. There are many methods to generate the graph pairs. The third method is $(2^{j}, b_j) = (2^{j}, (b_1)^{j} ($mod $N))$. The fourth method is: $(2^{j}, b_j) = (2^{(j-1)}. 2, b_{j-1}/2)$, if $b_{j-1}/2$ is an integer. Thus,$b_j$ can be found from any of the following formula and they are equal.
  
\begin{itemize}

 \item $N-a$ (if $a = (N-1)/2^{j}$ is an integer) 
 \item $(b_{j-1}.b_1)($mod $N)$
 \item  $(b_1)^{j} ($mod $N)$
 \item $b_{j-1}/2$ (if $b_{j-1}/2$ is an integer).   
\end{itemize} 
 \end{note}
\subsection{ILLUSTRATION}
Let us consider $N = 96577$. The step by step process and the calculations of the algorithm are explained in the Table~\ref{tab:3}.
\begin{table}[h]
\caption{Step by Step Executions of Algorithm 2}
\label{tab:3}      
\begin{tabular}{llll}
\hline\noalign{\smallskip}
$j$ & $(2^{j}, b_j)$ &	$D =|b_j-2^{j}|$	& $gcd(D,N)$\\
\noalign{\smallskip}\hline\noalign{\smallskip}
1 & (2, 48289) &	48287	& 1 \\
2 & (4, 72433) &	72429 &	1 \\
3 & (8, 84505)	& 84497 &	1\\
4 & (16, 90541)	& 90525 &	1\\
5 & (32, 93559)	& 93527 & 	1\\
6 & (64, 95068) & 95004 &		13\\

\noalign{\smallskip}\hline
\end{tabular}
\end{table}
\\ [2cm]
\textbf{Output:} \lq\lq$N$ is composite\rq\rq, $gcd(D,N)$ = 13 and $j$ = 6.

\subsection{CORRECTNESS OF THE ALGORITHM}
\begin{theorem}
\label{c1}
 If $(N-1)/2^{j}$ is an integer and $(N-1)/a = 2^{j}$, then $(2^{j}, N-a)$ is a graph pair for $N$.
\end{theorem}
\begin{proof}
 Given that $a = (N-1)/2^{j}$ , for $j \geq 1$.\\
$\Rightarrow N - a = \left[(2^{j}-1)N + 1\right]/2^{j}$.\\
To Prove: $(2^{j}, N-a)$ is a graph pair for $N$.\\
i.e., to prove: $2^{j}.(N-a) \equiv 1 ($mod $N)$.\\
 $2^{j}. (N- a) = 2^{j}. \left[(2^{j}-1)N + 1\right]/2^{j} = (2^{j}-1)N + 1 \equiv 1 ($mod $N)$ as $(2^{j}-1)N$ is divisible by $N$.
Thus, $(2^{j}, N-a)$ is a graph pair for $N$. $\Box$
\end{proof}
\begin{theorem}
\label{c2} 
$(2^{j-1}.2, (b_{j-1}.b_1)($mod $N))$ obtained  in step 4 is a graph pair for $N$.
\end{theorem}
\begin{proof} At the $j^{th}$ iteration, $(2, b_1)$  and $(2^{j-1}, b_{j-1})$ are graph pairs for $N$.\\
To Prove: $(2^{j-1}.2, (b_{j-1}.b_1)($mod $N))$  is also a graph pair for $N$ at the $j^{th}$ iteration.\\
$(2, b_1)$  and $(2^{j-1}, b_{j-1})$ are graph pairs for $N$.\\
$\Rightarrow 2.b_1 \equiv 1 ($mod $N)$ and $2^{j-1}. b_{j-1} \equiv 1($mod $N)$.\\
$\Rightarrow 2^{j}. b_1 b_{j-1} \equiv 1 ($mod $N)$ (by Theorem~\ref{product}).  \\
Thus, $(2^{j-1}.2, (b_{j-1}.b_1)($mod $N))$  is a graph pair for $N$ at the $j^{th}$ iteration. $\Box$
\end{proof}
\begin{note}
Similar to the Theorem ~\ref{c2}, it can be proved that $(2^{j}, (b_1)^{j} ($mod $N))$ is also a graph pair for $N$ and if $b_{j-1}/2$ is an integer, then $(2^{j}, b_j) = (2^{(j-1)}. 2, b_{j-1}/2)$ is also a graph pair for $N$.
\end{note}
\begin{theorem}
\label{c3}
 The prime divisor $k$ of $N$ is obtained in $(k-1)/2$ steps.
\end{theorem}
\begin{proof}
 Let $N$ be a multiple of $k$, where $k$ is prime.\\ 
To Prove: If $j = (k-1)/2$, then $gcd(D,N) = k$. \\
i.e., To Prove: If $j = (k-1)/2$, then $k | D$, where $D = |b_j-2^{j}|$.\\
 $b_j = (b_1)^{j}($mod $N)$ (from Note~\ref{other}) and $2^{j}$ can be written as  $2^{j}($mod $N)  = 2^{j}$ if $2^{j}< N$; otherwise it is the value of the remainder when $2^{j}$ is divided by $N$.\\
If $2^{j} ($mod $N) = 2^{j}$, \\then  $2^{j}-b_{j} = 2^{j}- [(b_1)^{j} ($mod $N)] = \left[2^{j}-(b_1)^{j}\right] ($mod $N)$.\\
If $2^{j} ($mod $N) = 2^{j}($mod $N)$, that is, the value of the remainder when $2^{j}$ is divided by $N$, then  $2^{j}-b_{j} = \left[2^{j}-(b_1)^{j}\right] ($mod $N)$.\\ Thus, whatever is the value of $2^{j}$, $2^{j}-b_{j} = \left[2^{j}-(b_1)^{j}\right] ($mod $N)$.\\
So, we have  to prove: If $j = (k-1)/2$, then $(2^{j}-(b_1)^{j}) \equiv 0($mod $k)$.\\
$b_1= N-(N-1)/2 = (N+1)/2$\\ 
$\Rightarrow (b_1)^{j} = \left[(N+1)/2\right]^{j}$.\\
$\Rightarrow (2^{j}-(b_1)^{j}) = 2^{j}- \left[(N+1)^{j}/2^{j}\right] =  \left[2^{2j} - (N+1)^{j}\right]/2^{j}$.\\
\textbf{Case(i):} $k = 7$.\\ 
Then  $j= (k-1) /2 = 3$ and  $\left[2^{2j} - (N+1)^{j}\right]/2^{j} = \left[64 -(N+1)^{3}\right]/8$.\\
To Prove: $\left[64 -(N+1)^{3}\right]/8 \equiv  0($mod $7)$.\\
i.e., to prove: $64 -(N+1)^{3} \equiv 0($mod $56)$ (by Theorem~\ref{cancel}).\\
$ 64 -(N+1)^{3} = 64-(N^{3}+3N^{2}+3N+1) = 63- N(N^{2}+3N+3)$.\\
As $k =7$, this implies that $N$ is a multiple of 7. So, let $N = 7m$, for every natural number $m$.\\ 
$\Rightarrow 64 -(N+1)^{3} =  63- 7m(49m^{2}+21m+3) = 7\left[9-m(49m^{2}+21m+3)\right]$. This expression has come out to be a multiple of 7.\\ So, in order to prove $64 -(N+1)^{3} \equiv 0($mod $56)$, it is enough to prove $9-m(49m^{2}+21m+3) \equiv 0($mod $8)$ (by Theorem~\ref{cancel}).\\
i.e., to prove: $m(49m^{2}+21m+3) \equiv 1 ($mod $8)$ (since $9 \equiv 1($mod $8)$).\\
As $N$ is odd,  $m$ is  odd.
\begin{itemize}
	\item $m=3 \Rightarrow m(49m^{2}+21m+3) = 1521 \equiv 1 ($mod $8)$.
\item $m=5 \Rightarrow m(49m^{2}+21m+3) = 6665 \equiv 1 ($mod $8)$.
\item $m=7 \Rightarrow m(49m^{2}+21m+3) = 17857 \equiv 1 ($mod $8)$.
\item $m=9 \Rightarrow m(49m^{2}+21m+3) = 37449 \equiv 1 ($mod $8)$.
\end{itemize}
 Continuing like this, it is easy to see that $m(49m^{2}+21m+3) \equiv 1 ($mod $8)$ and hence $(2^{j}-b_j) \equiv 0($mod $7)$.\\
\textbf{Case(ii):} $k = 11$. \\
Then  $j= (k-1) /2 = 5$ and $ \left[2^{2j} - (N+1)^{j}\right]/2^{j} = \left[1024 -(N+1)^{5}\right]/32$.\\
To Prove: $\left[1024 -(N+1)^{5}\right]/32 \equiv  0($mod $11)$.\\
i.e., to prove: $1024 -(N+1)^{5} \equiv 0($mod $352)$ (by Theorem~\ref{cancel}).
\begin{equation*}
\begin{split} 
1024 -(N+1)^{5} & = 1024-(N^{5}+5N^{4}+10N^{3}+10N^{2}+ 5N+1) \\
& = 1023- N(N^{4}+5N^{3}+10N^{2}+10N+5).
\end{split}
\end{equation*}

As $k =11$, this implies that $N$ is a multiple of 11. So, let $N = 11m$, for every natural number $m$.
\begin{equation*}
\begin{split} 
\Rightarrow 1024 -(N+1)^{5} & = 1023- 11m(14641 m^{4} + 6655m^{3}+1210m^{2}+110m+ 5) \\
& = 11\left[93- m(14641 m^{4}+6655m^{3}+1210m^{2}+110m+ 5)\right].
\end{split}
\end{equation*}
 This expression has come out to be a multiple of 11. \\So, in order to prove $1024 -(N+1)^{5} \equiv 0($mod $352)$, it is enough to prove  
$93- m(14641 m^{4}+6655m^{3}+1210m^{2}+110m+ 5) \equiv 0($mod $32)$ (by Theorem~\ref{cancel}).\\
i.e., to prove: $m(14641 m^{4}+6655m^{3}+1210m^{2}+110m+ 5) \equiv 29 ($mod $32)$ (since $93 \equiv 29($mod $32)$).\\
As $N$ is odd, $m$ is  odd.
\begin{itemize}
	\item $m=3 \Rightarrow \\m(14641 m^{4}+6655m^{3}+1210m^{2}+110m+ 5) = 4130493 \equiv 29 ($mod $32)$.
\item $m=5 \Rightarrow \\m(14641 m^{4}+6655m^{3}+1210m^{2}+110m+ 5) = 50066525 \equiv 29 ($mod $32)$.
\item $m=7 \Rightarrow \\m(14641 m^{4}+6655m^{3}+1210m^{2}+110m+ 5) = 262470397 \equiv 29 ($mod $32)$.
\item $m=9 \Rightarrow \\m(14641 m^{4}+6655m^{3}+1210m^{2}+110m+ 5) = 909090909 \equiv 29 ($mod $32)$.
\end{itemize}
Continuing like this, it is easy to see that \\$m(14641 m^{4}+6655m^{3}+1210m^{2}+110m+ 5) \equiv 29 ($mod $32)$ and hence \\$(2^{j}-b_j) \equiv 0($mod $11)$.\\
Proceeding like this, in general, for general $k$, $j = (k-1)/2$.\\
To prove:  $\left[2^{2j} - (N+1)^{j}\right]/2^{j}  \equiv 0($mod $k)$ if $j = (k-1)/2$.\\
i.e., to prove: $2^{2j} - (N+1)^{j} \equiv 0($mod $2^{j}k)$ (by Theorem~\ref{cancel}).

\begin{equation*}
\begin{split}
2^{2j} - (N+1)^{j} & =  2^{2j}-(N^{j}+jC_1N^{j-1}+jC_2N^{j-2}+\ldots + jN + 1) \\ 
& = (2^{2j}-1) - N(N^{j-1}+jC_1N^{j-2}+jC_2N^{j-3}+\ldots+ j).
\end{split}
\end{equation*}
As $N$ is a multiple  of $k$, let $N = km$. \\
So, 
\begin{equation}
\label{eq1}
2^{2j} - (N+1)^{j} = (2^{2j} -1)- km\left[(km)^{j-1}+jC_1(km)^{j-2}+jC_2(km)^{j-3}+\ldots+ j\right]
\end{equation} 
Here, $2^{2j} -1 = 2^{(k-1)} - 1 \equiv 0($mod $k)$ (since by Fermat's theorem-Theorem~\ref{Fermat}).\\
Thus, $(2^{2j} - 1)$ is also a multiple of $k$. So, let us write $(2^{2j} - 1) = km_1$.\\
Then, equation~\ref{eq1} becomes,  $2^{2j} - (N+1)^{j}\\ =  km_1 - km\left[(km)^{j-1}+jC_1(km)^{j-2}+jC_2(km)^{j-3}+\ldots+ j\right]\\ = k \left[m_1 - m((km)^{j-1}+jC_1(km)^{j-2}+jC_2(km)^{j-3}+\ldots+ j)\right]$.\\ This expression has come out to be a multiple of $k$.\\
So, in order to prove $2^{2j} - (N+1)^{j} \equiv 0($mod $2^{j}k)$ , it is enough to prove  
$m_1 - m\left[(km)^{j-1}+jC_1(km)^{j-2}+jC_2(km)^{j-3}+\ldots+ j\right] \equiv 0($mod $2^{j})$ (by Theorem~\ref{cancel}).\\
From case (i) and case (ii), we get for any odd $m$, this result holds true (Similar to the cases (i) and (ii), it can be easily checked for different values of $k$ and different values of $m$).
Thus, $(2^{j}-b_j) \equiv 0($mod $k)$, if $j =(k-1)/2$ and hence the prime divisor $k$ of $N$ is obtained in $(k-1)/2$ steps. $\Box$
\end{proof}
Next, a corollary is given, which may be helpful for the reduction of time complexity.
\begin{corollary}
\label{cor1}
 A prime divisor of $N$ can be obtained in $(L-1)/2$ steps, where $L$ is the least prime divisor of $N$.
\end{corollary}
\begin{proof}
Proof follows from the Theorem ~\ref{c3}. $\Box$ \end{proof}
\subsection{TIME COMPLEXITY OF THE ALGORITHM}
\label{time3}
Let us assume that $N > 2^{j}$ (as in step 7 of the algorithm). The time complexity of the algorithm is dependent on the following two steps:
\begin{itemize}
	\item Finding $gcd(D,N)$
\item How many times $gcd(D,N)$ is found, i.e., the value of $j$. 
\end{itemize}
Each computation of G.C.D  takes $O(\ln N)$~\cite{GCD}. As $j$ determines the number of steps, the total time complexity of the algorithm is $O(j \ln N)$.\\
\textbf{Computation of $j$:} By our assumption, $2^{j} < N < 2^{m}$, for some natural number $m$.\\ 
$\Rightarrow j < m$.\\
$\Rightarrow j =O(m) = O(\ln^{2}N)$ as $N < 2^{m}$.\\
Thus, the total time complexity of the algorithm is $O(\ln^{2}N)$, if $N > 2^{j}$. 
\subsection{WHAT HAPPENS WHEN $N < 2^{j}$}
Suppose if $N$ is composite and  $N < 2^{j}$, then the algorithm does not produce any output, because we get $gcd(D,N) = 1$ in all iterations. In this section, we give two methods, slightly modifying the algorithm given in the section~\ref{Algo3}.  
\subsubsection{First Method  when $N < 2^{j}$}
If $N < 2^{j}$ and $gcd(D,N)$ is still 1, then find the graph pair as $(2^{j} ($mod $N), b_j)$, where $b_j$ takes one of the forms given in Note~\ref{other} and continue the procedure until we get $gcd(D,N) \neq 1$. Steps of this method are as follows:\\
\textbf{Procedure 1}
\begin{enumerate}
	\item Put $j = j+1$.
\item Find $(2^{j} ($mod $N), b_j)$ and $D = |(b_j-2^{j}($mod $N))|$.
\item	If $gcd(D,N) = 1$, then goto step 1 
else print \lq\lq$N$ is composite\rq\rq, print $gcd(D,N)$ and stop.
\end{enumerate}
\textbf{Time Complexity of the Procedure 1}\\ 
We can see that for this case also, $j = (L-1)/2$ steps (by Corollary~\ref{cor1}). But, in this case, time complexity is not $O(\ln^{2}N)$ as in the previous case. Here, the time complexity depends on the least prime divisor $L$ of $N$, because we do the procedure for $j = (L-1)/2$ and $j$ is not less than $m$ in this case (since $2^{j} > N$). Usually the divisor of $N$ can be obtained in $O(N^{0.5})$. Thus, if $N< 2^{j}$, the time complexity of the algorithm is  $O(N^{0.5})$.
\subsubsection{Second Method  when $N < 2^{j}$}
In this section, we find a prime divisor of $N$ if $N < 2^{j}$ and also check up whether the number is prime or not. If $k|N$, then $k|N^{x}$, for any natural number $x$. So, by the Corollary~\ref{cor1}, divisor of $N^{x}$ is also obtained in $(L-1)/2$ steps, where $L$ is the least prime divisor of $\mathbf{N}$ (least prime divisor of $N$ and $N^{x}$ are same). We use this idea for the next procedure. First, let us find $N^{x}$ by finding suitable value of $x$. Store $N^{x}$ in $N_1$. Then apply the Algorithm 2 given in the section ~\ref{Algo3} for $N_1$, that is call GP($N_1$), with a slight modification in the procedure. In the Algorithm 2, we find $gcd(D,N)$ in step 8. But in this method, we do not find $gcd(D,N_1)$, instead we find $gcd(D,N)$, here also. Anyhow the answer will not be affected. The terminating condition is changed here. This method checks for primality also. \\
\textbf{Procedure 2}
\begin{enumerate}
\item Find $N^{x}$ such that  $N^{x} < 2^{N^{0.5}}$ and let $N^{x} = N_1$.
\item Let $j = 0$.
\item Put $j = j+1$.
\item If $(N_1-1)/2^{j}$ is an integer, then find the graph pair $((N_1-1)/a, N_1-a)$ such that $a = (N_1-1)/2^{j}$.
\item else find the graph pair $(2^{(j-1)}. 2, (b_{j-1}.b_1)($mod $N_1))$.
\item	Denote this graph pair as $(2^{j}, b_j)$.
\item	Find $|b_j- 2^{j}|$ and store it in $D$.
\item If $j < N^{0.5}$, then GOTO step 9 else Print \lq\lq$N$ is prime\rq\rq\  and stop.
\item If $gcd(D,N) = 1$, then GOTO step 3 else Print \lq\lq$N$ is composite\rq\rq, print $gcd(D,N)$ and $j$ and stop.
\end{enumerate}
\textbf{Time Complexity of Procedure 2}\\
As discussed in the section~\ref{time3}, the time complexity of the Procedure 2 is $O(j \ln N)$. What is the value of $j$ here?  $j$ denotes the number of times $gcd(D,N)$ is found out in the Procedure 2 (that is, the number of iterations of the Procedure 2).  By Corollary~\ref{cor1}, $j = (L-1)/2 = O(L)$, where $L$ is the least prime divisor of $N$. Also, any divisor of $N$ will be obtained in at most $O(N^{0.5})$. That is why we made the terminating condition as $j < N^{0.5}$.  Thus, the total time complexity of the Procedure 2 for finding a prime divisor of $N$ or for primality testing is at most $O(N^{0.5} \ln N)$, which is higher than the usual time complexity $O(N^{0.5})$.
\begin{note}
Comparing the time complexity of the Procedure 1 and the time complexity of the Procedure 2, we can conclude that if $N < 2^{j}$, then the Procedure 1 can be applied rather than the Procedure 2  and hence the time complexity to find a prime divisor of $N$ is at most $O(N^{0.5})$ (usual time complexity).

For primality test also, the time complexity is $O(N^{0.5}\ln N)$, which is too high. The Procedure 2 need not be used. The procedure 2 is given only to know whether Algorithm 2 can be used to check up primality or not. But, we got higher time complexity. So, we make the following open problem.
\end{note}
\textbf{Open Problem 1:} Use Algorithm 2 to check up whether $N$ is prime or not in lesser time.
\subsection{ADVANTAGES AND DISADVANTAGES OF THE ALGORITHM 2}
In this section, we list out some advantages and disadvantages of the Algorithm 2 simultaneously. First we mention a disadvantage and then we mention the advantage out of it or how the disadvantage is overcome. Let the term  \lq Disadv\rq\  denote disadvantage and the term \lq Adv\rq\  denote the advantage.\\
\textbf{Disadv:} The algorithm finds a prime divisor for $N$ in $O(\ln^{2}N)$ only if $N > 2^{j}$, where $j = (L-1)/2$ and $L$ is the least prime divisor of $N$ and in $O(N^{0.5})$ if $N < 2^{j}$. \\
\textbf{Adv:} Although this algorithm gives a prime divisor for $N$ in $O(\ln^{2}N)$ only if $N > 2^{j}$, it gives result for infinitely many cases and only finite cases are left, which are countable.
For example, 
\begin{itemize}
	\item If $N$ is a multiple of 7, then by this algorithm, $j = 3$ and hence $N > 2^{3} = 8$. So, this algorithm finds divisor for $N$ within 3 steps (in $O(\ln^{2}N)$) for all $N > 8$, and hence infinite cases are covered.
\item Similarly, if $N$ is a multiple of 11, then by this algorithm, $j = 5$ and hence $N > 2^{5} = 32$. So, this algorithm finds divisor for $N$ within 5 steps for all $N > 32$  (in $O(\ln^{2}N)$), and hence infinite cases are covered.
\item Similarly if $N$ is a multiple of 101, then by this algorithm, $j = 50$ and hence $N > 2^{50}$ . So, this algorithm finds divisor for $N$ within 50 steps  (in $O(\ln^{2}N)$) for all $N > 2^{50}$ and hence infinite cases are covered  and so on.
\end{itemize}
So, we can say that the algorithm gives good results for \textbf{infinitely many cases} (or approximately large $N$).\\
\textbf{Disadv:} Only if $j$ is known, this algorithm gives prime divisor of $N$ in lesser time. But $j$ is in turn dependent on $L$, which is the least prime divisor of $N$ and hence this process ends in a cycle. \\
\textbf{Adv:} That is why we say that the algorithm gives good result for approximately large $N$.\\ 
\textbf{Disadv:} The time complexity of the algorithm to find a prime divisor of $N$, if $N < 2^{j}$ is at most $O(N^{0.5})$, which is too slow. \\
\textbf{Adv:} Although the process is too slow, the coding is very very simple. Also the proof of correctness of the algorithm and time complexity are also simple.\\
\textbf{Disadv:} We discussed in Theorem~\ref{c3} that $j = (k-1)/2$. That is,
\begin{itemize}
	\item Multiple of 7 will be obtained in 3 steps;
\item Multiple of 11 will be obtained in 5 steps;
\item Multiple of 13 will be obtained in 6 steps;
\item Multiple of 17 will be obtained in 8 steps and so on.

\end{itemize}
 
But usually when a number $N$ is given, in order to find its divisor, we will do trial and error method. We will start checking from the primes 3, 7, 11, $\ldots$ (in order).
So, if we check in the primes in their order, then 
\begin{itemize}
	\item Multiple of 7 will be obtained in 2 steps as 7 is the second odd prime;
\item Multiple of 11 will be obtained in 3 steps as 11 is the third odd prime;
\item Multiple of 13 will be obtained in 4 steps as 13 is the fourth odd prime; and so on.
\end{itemize}
The following table, Table~\ref{tab:4}, gives the comparison of our method and the usual method. Let $y$ (say) be the least prime divisor of $N$.
\begin{table}[h]
\caption{Comparison between Our Method and Trial $\&$  Error Method}
\label{tab:4}      
\begin{tabular}{lll}
\hline\noalign{\smallskip}
 $y$ &	No. of steps from our method 	& No. of steps from trial and error method\\
\noalign{\smallskip}\hline\noalign{\smallskip}
7 &	3 &	2\\
11 &	5 &	3\\
13	& 6	 & 4\\
17 &	8	& 5\\
19 &	9 &	6 and so on.\\

\noalign{\smallskip}\hline
\end{tabular}
\end{table}

So, we can observe from the table that our method takes more steps (time) than trial and error method. But can any one tell the exact position of the largest prime? Only if the position  is known, trial and error method is good. So, in such a situation, our algorithm gives better result.\\
\textbf{Adv:} As the exact position of prime cannot be determined, our method can be used to get better results in $O(\ln^{2}N)$ time.\\
\textbf{Adv:} Given  a prime $N$, this algorithm verifies it in polynomial time. 

From the above disadvantages and advantages, we infer that
\begin{itemize}
	\item Given any $N$, the algorithm does not find the divisor in $O(\ln^{2}N)$.
\item  The computation takes $O(N^{0.5})$, which is too slow. 
\item But coding the Algorithm 2 is very very simple. 
\item The proof of correctness  and time complexity of the algorithm are also simple.
\item The result depends on the input $N$. If, (by luck), we are given the appropriate $N$  as input, that is, $N > 2^{j}$, then the algorithm solves the problem in $O(\ln^{2}N)$, which is better  than all other existing algorithms. For example, if $N$ is very large and it satisfies $N > 2^{j}$ (by luck), then such a case is solved by our algorithm in lesser time.
\item  So, P(getting a prime divisor of $N| N > 2^{j}$) = 1 and \\P(getting a prime divisor of any $N$) = 0.5. 
\end{itemize}
\section{Conclusion}
In this paper, two approximation algorithms are given, which  find a divisor of the given number.

The first algorithm is an interesting approximation algorithm and is given using a new approach.The algorithm finds a divisor of the given number.  The intricate portions of the algorithm are discussed and a conjecture is given regarding the time complexity of the algorithm.The algorithm can be applied to find the next largest Mersenne prime number.

The second algorithm is a simple approximation algorithm, given using the concept of graph pairs. The algorithm finds a prime divisor of the given number $N$ in $O(\ln^{2}N)$, for infinitely many cases (approximately large $N$). Time complexity and correctness of the algorithm are proved. The advantages and disadvantages of the algorithm are also discussed. An attempt is made to use the Algorithm 2 for primality test and an open problem is given regarding this.

The authors feel that this paper may be helpful for further research as many interesting things are yet to be analyzed and it may pave way for new ideas.
\section{Future Work}
We are  working on the conjectures and open problems given this paper to develop the concepts further.

%\begin{acknowledgements}
%\end{acknowledgements}


\begin{thebibliography}{13}
\bibitem{NT} {T. M. Apostol}, {\em Introduction to Analytic Number Theory}, {Springer-Verlag}, 1997.
\bibitem{Mill} {G. L. Miller}, {\em Riemann's hypothesis and tests for primality}, {J. Comput. Sys. Sci.}, 13:300-317, 1976.
\bibitem{Rabin} {M. O. Rabin}, {\em  Probabilistic algorithm for testing primality}, {J. Number Theory}, 12: 128-138, 1980.
\bibitem{SS} {R. Solovay and V. Strassen}, {\em A fast Monte-Carlo test for primality}, {SIAM Journal on Computing}, 6:84-86, 1977.
\bibitem{APR} {L. M. Adleman, C. Pomerance, and R. S. Rumely}, {\em On distinguishing prime numbers from composite numbers}, {Ann. Math.}, 117:173-206, 1983.
\bibitem{GK} {S. Goldwasser and J. Kilian}, {\em  Almost all primes can be quickly certified}, {In Proceedings of Annual ACM Symposium on the Theory of Computing}, 316-329, 1986.
\bibitem{ATK} {A. O. L. Atkin}, {\em Lecture notes of a conference, boulder (colorado)}, {Manuscript}, August 1986.
\bibitem{Huang} {L. M. Adleman and M.-D. Huang}, {\em Primality testing and two dimensional Abelian varieties over finite fields}, {Lecture Notes in Mathematics}, 1512, 1992.
\bibitem{AKS} {Manindra Agrawal, Neeraj Kayal and Nitin Saxena}, {\em PRIMES is in P}, {Annals of Mathematics} 160, no. 2, pp. 781-793, 2004.
\bibitem{GCD} {J.v.z. Gathen and J.Gerhard}, {\em Modern Computer Algebra}, Cambridge University Press, 1999.

\bibitem{cycle} {Sabra S.Anderson}, {\em Graph Theory and finite combinatorcs}, {Markham Publishing Company, Chicago}, 1970.
\bibitem{TNJ1} {T.N.Janakiraman and M.M.D.Boominathan}, {\em On Graph Pairs}, {In Proceedings of International Conference on Recent Advances in Mathematical Sciences and Symposium on Challenges in Mathematical Sciences for the New Millennium, Department of Mathematics, IIT,Kharagpur, India}, December 20-22, 2009.
\bibitem{TNJ} {M. Marcus Diepen Boominathan}, {\em Some Regular Structures in graphs and related concepts}, Thesis, December 2000.

\end{thebibliography}
\end{document}